\documentclass[12pt]{article}
\usepackage{amssymb}
\usepackage{amsthm}
\usepackage{amsmath}
\usepackage{graphicx}
\usepackage{enumerate}

\newtheorem{theorem}{Theorem}%[section]
\newtheorem{definition}[theorem]{Definition}
\newtheorem{lemma}[theorem]{Lemma}

\newtheorem{remark}[theorem]{Remark}

\title{Residuation in lattice effect algebras}
\author{Ivan~Chajda and Helmut~L\"anger}
\date{}
\begin{document}
\footnotetext[1]{Support of the research by \"OAD, project CZ~02/2019, and support of the research of the first author by IGA, project P\v rF~2019~015, is gratefully acknowledged.}
\maketitle
\begin{abstract}
We introduce the concept of a quasiresiduated lattice and prove that every lattice effect algebra can be organized into a commutative quasiresiduated lattice with divisibility. Also conversely, every such lattice can be converted into a lattice effect algebra and every lattice effect algebra can be reconstructed form its assigned quasiresiduated lattice. We apply this method also for lattice pseudoeffect algebras introduced recently by Dvure\v censkij and Vetterlein. We show that every good lattice pseudoeffect algebra can be organized into a (possibly non-commutative) quasiresiduated lattice with divisibility and conversely, every such lattice can be converted into a lattice pseudoeffect algebra. Moreover, also a good lattice pseudoeffect algebra can be reconstructed from the assigned quasiresiduated lattice.
\end{abstract}
 


In order to axiomatize quantum logic effects in a Hilbert space, Foulis and Bennett (\cite{FB}) introduced the so-called effect algebras. These are partial algebras with one partial binary operation which can be converted into bounded posets in general and into lattices in particular cases. It turns out that effect algebras form a successful axiomatization of the logic of quantum mechanics, but we suppose that there exists a connection with a kind of substructural logics whose algebraic semantics is based on residuated lattices. An attempt in this direction was already done in \cite{CH} where the so-called conditional residuation was introduced. A disadvantage of this approach is that the axioms of residuated structures are reflected only in the case when the terms used in adjointness are defined. This is an essential restriction which prevents the development of this theory. The aim of the present paper is to introduce the more general concept of quasiresiduation and to show that lattice effect algebras and lattice pseudoeffect algebras satisfy this concept. Pseudoeffect algebras were introduced recently by Dvure\v censki and Vetterlein (\cite{DV}).

We start with the following definition.

\begin{definition}
An {\em effect algebra} is a partial algebra $\mathbf E=(E,+,{}',0,1)$ of type $(2,1,0,0)$ where $(E,{}',0,1)$ is an algebra and $+$ is a partial operation satisfying the following conditions for all $x,y,z\in E$:
\begin{enumerate}[{\rm(E1)}]
\item $x+y$ is defined if and only if so is $y+x$ and in this case $x+y=y+x$,
\item $(x+y)+z$ is defined if and only if so is $x+(y+z)$ and in this case $(x+y)+z=x+(y+z)$,
\item $x'$ is the unique $u\in E$ with $x+u=1$,
\item if $1+x$ is defined then $x=0$.
\end{enumerate}
\end{definition}

On $E$ a binary relation $\leq$ can be defined by
\[
x\leq y\text{ if there exists some }z\in E\text{ with }x+z=y
\]
($x,y\in E$). Then $(E,\leq,0,1)$ is a bounded poset and $\leq$ is called the {\em induced order} of $\mathbf E$. If $(E,\leq)$ is a lattice then $\mathbf E$ is called a {\em lattice effect algebra}.

In the sequel we will use the properties of effect algebras listed in the following lemma.

\begin{lemma}\label{lem1}
{\rm(}see {\rm\cite{DP},\cite{DV})} If $\mathbf E=(E,+,{}',0,1)$ is an effect algebra, $\leq$ its induced order and $a,b,c,d\in E$ then the following hold:
\begin{enumerate}[{\rm(i)}]
\item $a''=a$,
\item $a\leq b$ implies $b'\leq a'$,
\item $a+b$ is defined if and only if $a\leq b'$,
\item if $a\leq b$ and $b+c$ is defined then $a+c$ is defined and $a+c\leq b+c$,
\item if $a\leq b$ then $a+(a+b')'=b$,
\item $a+0=0+a=a$,
\item $0'=1$ and $1'=0$.
\end{enumerate}
\end{lemma}

A {\em partial monoid} is a partial algebra $\mathbf A=(A,\odot,1)$ of type $(2,0)$ where $1\in A$ and $\odot$ is a partial operation satisfying the following conditions for all $x,y,z\in A$:
\begin{enumerate}[{\rm(i)}]
\item $(x\odot y)\odot z$ is defined if and only if so is $x\odot(y\odot z)$ and in this case $(x\odot y)\odot z=x\odot(y\odot z)$,
\item $x\odot1=1\odot x=x$.
\end{enumerate}
The {\em partial monoid} $\mathbf A$ is called {\em commutative} if it satisfies the following condition for all $x,y\in A$:
\begin{enumerate}
\item[(iii)] $x\odot y$ is defined if and only if so is $y\odot x$ and in this case $x\odot y=y\odot x$,
\end{enumerate}

The authors already introduced a certain modification of residuation for sectionally pseudocomplemented lattices, see \cite{CL}. For lattice effect algebras, we introduce another version of residuation called quasiresiduation.

\begin{definition}\label{def1}
A {\em commutative quasiresiduated lattice} is a partial algebra $\mathbf C=(C,\vee,\wedge,$ $\odot,\rightarrow,0,1)$ of type $(2,2,2,2,0,0)$ where $(C,\vee,\wedge,0,1)$ is a bounded lattice, $\odot$ is a partial and $\rightarrow$ a full operation satisfying the following conditions for all $x,y,z\in C$:
\begin{enumerate}[{\rm(C1)}]
\item $(C,\odot,1)$ is a partial commutative monoid where $x\odot y$ is defined if and only if $x'\leq y$,
\item $x''=x$, and $x\leq y$ implies $y'\leq x'$,
\item $(x\vee y')\odot y\leq y\wedge z$ if and only if $x\vee y'\leq y\rightarrow z$.
\end{enumerate}
Here $x'$ is an abbreviation for $x\rightarrow0$. The {\em commutative quasiresiduated lattice} $\mathbf C$ is called {\em divisible} if
\[
x\leq y\text{ implies }y\odot(y\rightarrow x)=x
\]
for all $x,y\in C$.
\end{definition}

Note that the terms in (C3) are everywhere defined.

In case $y'\leq x$ and $z\leq y$ condition (C3) reduces to
\[
x\odot y\leq z\text{ if and only if }x\leq y\rightarrow z
\]
which is usual adjointness. Therefore condition (C3) will be called {\em commutative quasiadjointness}. Hence, contrary to the similar concept in \cite{CH}, in commutative quasiadjointness we have only everywhere defined terms in $\mathbf C$ although $\mathbf C$ is a partial algebra.

Our aim is to show that every lattice effect algebra can be organized into a commutative quasiresiduated lattice.

\begin{theorem}\label{th1}
Let $\mathbf E=(E,+,{}',0,1)$ be a lattice effect algebra with lattice operations $\vee$ and $\wedge$ and put
\begin{align*}
      x\odot y & :=(x'+y')'\text{ if and only if }x'\leq y, \\
x\rightarrow y & :=(x\wedge y)+x'
\end{align*}
{\rm(}$x,y\in E${\rm)}. Then $\mathbb C(\mathbf E):=(E,\vee,\wedge,\odot,\rightarrow,0,1)$ is a divisible commutative quasiresiduated lattice.
\end{theorem}

\begin{proof}
Let $a,b,c\in E$. Obviously, $(E,\vee,\wedge,0,1)$ is a bounded lattice and (C1) and (C2) hold. If $(a\vee b')\odot b\leq b\wedge c$ then $((a\vee b')'+b')'\leq b\wedge c$ and $b'\leq a\vee b'$ and hence
\[
a\vee b'=b'+((a\vee b')'+b')'\leq b'+(b\wedge c)=b\rightarrow c.
\]
If, conversely, $a\vee b'\leq b\rightarrow c$ then $a\vee b'\leq(b\wedge c)+b'$ and $b'\leq(b\wedge c)'$ and hence
\[
(a\vee b')\odot b=(b'+(a\vee b')')'\leq(b'+((b\wedge c)+b')')'=(b\wedge c)''=b\wedge c
\]
proving (C3). If $a\leq b$ then $b'\leq a'$ and hence
\[
b\odot(b\rightarrow a)=(b'+(a+b')')'=a''=a
\]
proving divisibility.
\end{proof}

\begin{remark}
Let us mention that Definition~\ref{def1} can be modified in such a way that it contains only everywhere defined operations. Namely, if we put
\[
x\otimes y:=(x\vee y')\odot y
\]
for all $x,y\in C$ then $\otimes$ is everywhere defined and satisfies the identities $x\otimes1\approx1\otimes x\approx x$, and commutative quasiadjointness can then be expressed in the form
\[
x\otimes y\leq y\wedge z\text{ if and only if }x\vee y'\leq y\rightarrow z.
\]
This means that our definition of commutative quasiresiduation differs from that of usual residuation only in the point that $y$ occurs on the right-hand side of $x\otimes y\leq y\wedge z$ and $y'$ on the left-hand side of $x\vee y'\leq y\rightarrow z$. On the other hand, using this version, divisibility cannot be easily defined. Moreover, since in lattice effect algebras we have
\[
y\rightarrow z=(y\wedge z)+y'=y\rightarrow(y\wedge z),
\]
commutative quasiadjointness can be rewritten in the form
\[
(x\vee y')\odot y\leq y\wedge z\text{ if and only if }x\vee y'\leq y\rightarrow(y\wedge z)
\]
which corresponds to usual adjointness if we abbreviate $x\vee y'$ by $X$ and $y\wedge z$ by $Z$, i.e.
\[
X\odot y\leq Z\text{ if and only if }X\leq y\rightarrow Z.
\]
\end{remark}

We can prove also the converse.

\begin{theorem}
Let $\mathbf C=(C,\vee,\wedge,\odot,\rightarrow,0,1)$ be a commutative quasiresiduated lattice and put
\begin{align*}
x+y & :=(x'\odot y')'\text{ if and only if }x\leq y', \\
 x' & :=x\rightarrow0
\end{align*}
{\rm(}$x,y\in C${\rm)}. Then $\mathbb E(\mathbf C):=(C,+,{}',0,1)$ is a lattice effect algebra whose order coincides with that in $\mathbf C$.
\end{theorem}

\begin{proof}
Let $a,b\in C$. It is easy to see that (E1), (E2) and (E4) hold. Since
\[
0\vee a'=a'\leq a'=a\rightarrow0
\]
we have
\[
a\odot a'=a\odot(0\vee a')\leq a\wedge0=0,
\]
i.e.\ $a\odot a'=0$. If, conversely, $a\odot b=0$ then $a'\leq b$ and hence
\[
a\odot(b\vee a')=a\odot b=0\leq a\wedge0
\]
whence
\[
b=b\vee a'\leq a\rightarrow0=a'
\]
showing $b=a'$. Hence $a\odot b=0$ if and only if $b=a'$. Now the following are equivalent:
\begin{align*}
       a+b & =1, \\
a'\odot b' & =0, \\
         a & =b', \\
         b & =a'.
\end{align*}
This shows (E3). Moreover, the following are equivalent:
\begin{align*}
& a\leq b\text{ holds in }\mathbb E(\mathbf R), \\
& a+b'\text{ is defined}, \\
& a'\odot b\text{ is defined}, \\
& a\leq b\text{ holds in }\mathbf R.
\end{align*}
Since $(C,\vee,\wedge)$ is a lattice and the partial order relations in $\mathbf C$ and $\mathbb E(\mathbf C)$ coincide, $\mathbb E(\mathbf C)$ is a lattice effect algebra.
\end{proof}

Moreover, every lattice effect algebra can be reconstructed from the assigned quasiresiduated lattice as shown in the following result.

\begin{theorem}
Let $\mathbf E$ be a lattice effect algebra. Then $\mathbb E(\mathbb C(\mathbf E))=\mathbf E$.
\end{theorem}

\begin{proof}
Let
\begin{align*}
                      \mathbf E & =(E,+,{}',0,1)\text{ with lattice operations }\vee\text{ and }\wedge, \\
           \mathbb C(\mathbf E) & =(E,\vee,\wedge,\odot,\rightarrow,0,1), \\
\mathbb E(\mathbb C(\mathbf E)) & =(E,\oplus,{}^*,0,1)
\end{align*}
and $a,b\in E$. Then
\[
a^*=a\rightarrow0=(a\wedge0)+a'=0+a'=a'.
\]
Moreover, the following are equivalent:
\begin{align*}
& a\oplus b\text{ is defined}, \\
& a\leq b'\text{ in }\mathbb C(\mathbf E), \\
& a\leq b'\text{ in }\mathbf E
\end{align*}
and in this case
\[
a\oplus b=(a^*\odot b^*)^*=(a'\odot b')'=(a''+b'')''=a+b.
\]
\end{proof}

Now we turn our attention to a more general case. The following concept was introduced by Dvure\v censkij and Vetterlein (\cite{DV}).

\begin{definition}
A {\em pseudoeffect algebra} is a partial algebra $\mathbf P=(P,+,\,\bar{}\,,\,\tilde{}\,,0,1)$ of type $(2,1,1,0,0)$ where $(P,\,\bar{}\,,\,\tilde{}\,,0,1)$ is an algebra and $+$ is a partial operation satisfying the following conditions for all $x,y,z\in P$:
\begin{enumerate}[{\rm(P1)}]
\item If $x+y$ is defined then there exist $u,w\in P$ with $u+x=y+w=x+y$,
\item $(x+y)+z$ is defined if and only if $x+(y+z)$ is defined, and in this case $(x+y)+z=x+(y+z)$,
\item $\bar x$ is the unique $u\in P$ with $u+x=1$, and $\tilde x$ is the unique $w\in P$ with $x+w=1$,
\item if $1+x$ or $x+1$ is defined then $x=0$.
\end{enumerate}
The {\em pseudoeffect algebra} $\mathbf P$ is called {\em good} if $\widetilde{\bar x+\bar y}=\overline{\tilde x+\tilde y}$ for all $x,y\in P$ with $\tilde x\leq y$.
\end{definition}

On $P$ a binary relation $\leq$ can be defined by
\[
x\leq y\text{ if there exists some }z\in E\text{ with }x+z=y
\]
($x,y\in P$). Then $(P,\leq,0,1)$ is a bounded poset and $\leq$ is called the {\em induced order} of $\mathbf P$. If $(P,\leq)$ is a lattice then $\mathbf P$ is called a {\em lattice pseudoeffect algebra}.

For our investigation we need the following results taken from \cite{DV}.

\begin{lemma}
If $\mathbf P=(P,+,\,\bar{}\,,\,\tilde{}\,,0,1)$ is a pseudoeffect algebra, $\leq$ its induced order and $a,b,c\in P$ then
\begin{enumerate}[{\rm(i)}]
\item $\bar{\tilde a}=\tilde{\bar a}=a$,
\item the following are equivalent: $a\leq b$, $\bar b\leq\bar a$, $\tilde b\leq\tilde a$,
\item $a+b$ is defined if and only if $a\leq\bar b$,
\item if $a\leq b$ and $b+c$ is defined then $a+c$ is defined and $a+c\leq b+c$,
\item if $a\leq b$ and $c+b$ is defined then $c+a$ is defined and $c+a\leq c+b$,
\item if $a\leq b$ then $a+\widetilde{\bar b+a}=\overline{a+\tilde b}+a=b$,
\item $a+0=0+a=a$,
\item $\bar0=\tilde0=1$ and $\bar1=\tilde1=0$,
\item the following are equivalent: $a\leq b$, there exists some $d\in P$ with $a+d=b$, there exists some $e\in P$ with $e+a=b$.
\end{enumerate}
\end{lemma}

Since pseudoeffect algebras are more general than effect algebras, we must define qua\-si\-re\-si\-dua\-ted lattice for the case when the partial operation $\odot$ is not commutative and the mapping $x\mapsto\bar x$ is not an involution.

\begin{definition}
A {\em quasiresiduated lattice} is a partial algebra $\mathbf Q=(Q,\vee,\wedge,\odot,\rightarrow,\leadsto,0,1)$ of type $(2,2,2,2,2,0,0)$ where 
$(Q,\vee,\wedge,0,1)$ is a bounded lattice, $\odot$ is a partial and $\rightarrow$ and $\leadsto$ are full operations satisfying the following conditions for all $x,y,z\in Q$:
\begin{enumerate}[{\rm(Q1)}]
\item $(Q,\odot,1)$ is a partial monoid where $x\odot y$ is defined if and only if $\tilde x\leq y$,
\item $\tilde{\bar x}=\bar{\tilde x}=x$, and $x\leq y$ implies $\bar y\leq\bar x$ and $\tilde y\leq\tilde x$,
\item $(x\vee\bar y)\odot y\leq y\wedge z$ if and only if $x\vee\bar y\leq y\rightarrow z$,
\item $y\odot(x\vee\tilde y)\leq y\wedge z$ if and only if $x\vee\tilde y\leq y\leadsto z$,
\item $\widetilde{\bar x\odot\bar y}=\overline{\tilde x\odot\tilde y}$.
\end{enumerate}
Here $\bar x$ and $\tilde x$ are abbreviations for $x\rightarrow0$ and $x\leadsto0$, respectively. The quasiresiduated lattice $\mathbf Q$ is called {\em divisible} if
\[
x\leq y\text{ implies }(y\rightarrow x)\odot y=y\odot(y\leadsto x)=x
\]
for all $x,y\in Q$.
\end{definition}

Note that the terms in (Q3) and (Q4) are everywhere defined.

In case $\bar y\leq x$ and $z\leq y$ condition (Q3) reduces to
\[
x\odot y\leq z\text{ if and only if }x\leq y\rightarrow z
\]
which is usual adjointness. Analogously, in case $\tilde y\leq x$ and $z\leq y$ condition (Q4) reduces to
\[
y\odot x\leq z\text{ if and only if }x\leq y\rightarrow z
\]
which is usual adjointness if $\odot$ is commutative. Therefore conditions (Q3) and (Q4) will be called {\em quasiadjointness}. Hence, contrary to the similar concept in \cite{CH}, in quasiadjointness we have only everywhere defined terms in $\mathbf Q$ although $\mathbf Q$ is a partial algebra.

Similarly as for effect algebras, we prove that every good lattice pseudoeffect algebra can be organized into a quasiresiduated lattice which, however, need not be commutative.

\begin{theorem}
Let $\mathbf P=(P,+,\,\bar{}\,,\,\tilde{}\,,0,1)$ be a good lattice pseudoeffect algebra with lattice operations $\vee$ and $\wedge$ and put
\begin{align*}
      x\odot y & :=\widetilde{\bar x+\bar y}\text{ if and only if }\tilde x\leq y, \\
x\rightarrow y & :=\bar x+(x\wedge y),\;x\leadsto y:=(x\wedge y)+\tilde x
\end{align*}
{\rm(}$x,y\in P${\rm)}. Then $\mathbb Q(\mathbf P):=(P,\vee,\wedge,\odot,\rightarrow,\leadsto,0,1)$ is a divisible quasiresiduated lattice.
\end{theorem}

\begin{proof}
Let $a,b,c\in E$. Obviously, $(P,\vee,\wedge,0,1)$ is a bounded lattice and (Q1), (Q2) and (Q5) hold. If $(a\vee\bar b)\odot b\leq b\wedge c$ then $\widetilde{\overline{a\vee\bar b}+\bar b}\leq b\wedge c$ and $\bar b\leq a\vee\bar b$ and hence
\[
a\vee\bar b=\bar b+\widetilde{\overline{a\vee\bar b}+\bar b}\leq\bar b+(b\wedge c)=b\rightarrow c.
\]
If, conversely, $a\vee\bar b\leq b\rightarrow c$ then $a\vee\bar b\leq\bar b+(b\wedge c)$ and $\bar b\leq\overline{b\wedge c}$ and hence
\[
(a\vee\bar b)\odot b=\widetilde{\overline{a\vee\bar b}+\bar b}\leq\widetilde{\overline{\bar b+(b\wedge c)}+\bar b}=\widetilde{\overline{b\wedge c}}=b\wedge c
\]
roving (Q3). If $b\odot(a\vee\tilde b)\leq b\wedge c$ then $\overline{\tilde b+\widetilde{a\vee\tilde b}}\leq b\wedge c$ and $\tilde b\leq a\vee\tilde b$ and hence
\[
a\vee\tilde b=\overline{\tilde b+\widetilde{a\vee\tilde b}}+\tilde b\leq(b\wedge c)+\tilde b=b\leadsto c.
\]
If, conversely, $a\vee\tilde b\leq b\leadsto c$ then $a\vee\tilde b\leq(b\wedge c)+\tilde b$ and $\tilde b\leq\widetilde{b\wedge c}$ and hence
\[
b\odot(a\vee\tilde b)=\overline{\tilde b+\widetilde{a\vee\tilde b}}\leq\overline{\tilde b+\widetilde{(b\wedge c)+\tilde b}}=\overline{\widetilde{b\wedge c}}=b\wedge c
\]
proving (Q4). If $a\leq b$ then $\bar b\leq\bar a$ and $\tilde b\leq\tilde a$ and hence
\begin{align*}
(b\rightarrow a)\odot b & =\widetilde{\mathit{\overline{\bar b+a}+\bar b}}=\tilde{\bar a}=a, \\
    b\odot(b\leadsto a) & =\overline{\tilde b+\widetilde{a+\tilde b}}=\bar{\tilde a}=a
\end{align*}
proving divisibility.
\end{proof}

We can prove also the converse.

\begin{theorem}
Let $\mathbf Q=(Q,\vee,\wedge,\odot,\rightarrow,\leadsto,0,1)$ be a quasiresiduated lattice and put
\begin{align*}
     x+y & :=\widetilde{\bar x\odot\bar y}\text{ if and only if }x\leq\bar y, \\
  \bar x & :=x\rightarrow0,\;\tilde x:=x\leadsto0
\end{align*}
{\rm(}$x,y\in Q${\rm)}. Then $\mathbb P(\mathbf Q):=(Q,+,\,\bar{}\,,\,\tilde{}\,,0,1)$ is a good lattice pseudoeffect algebra whose order coincides with that in $\mathbf Q$.
\end{theorem}

\begin{proof}
Let $a,b,c\in Q$. It is easy to see that (E2) and (E4) hold. Since
\[
0\vee\tilde a=\tilde a\leq\tilde a=a\leadsto0
\]
we have
\[
a\odot\tilde a=a\odot(0\vee\tilde a)\leq a\wedge0=0,
\]
i.e.\ $a\odot\tilde a=0$. If, conversely, $a\odot b=0$ then $\tilde a\leq b$ and hence
\[
a\odot(b\vee\tilde a)=a\odot b=0\leq a\wedge0
\]
whence
\[
b=b\vee\tilde a\leq a\leadsto0=\tilde a
\]
showing $b=\tilde a$. Hence $a\odot b=0$ if and only if $b=\tilde a$. Since
\[
0\vee\bar a=\bar a\leq\bar a=a\rightarrow0
\]
we have
\[
\bar a\odot a=(0\vee\bar a)\odot a\leq a\wedge0=0,
\]
i.e.\ $\bar a\odot a=0$. If, conversely, $b\odot a=0$ then $\tilde b\leq a$, i.e.\ $\bar a\leq b$, and hence
\[
(b\vee\bar a)\odot a=b\odot a=0\leq a\wedge0
\]
whence
\[
b=b\vee\bar a\leq a\rightarrow0=\bar a
\]
showing $b=\bar a$. Hence $b\odot a=0$ if and only if $b=\bar a$. Now the following are equivalent:
\begin{align*}
              a+b & =1, \\
\bar a\odot\bar b & =0, \\
                a & =\bar b, \\
                b & =\tilde a.
\end{align*}
This shows (P3). Since
\[
a\odot(1\vee\tilde a)=a\odot1=a\leq a\wedge a
\]
we have
\[
1=1\vee\tilde a\leq a\leadsto a,
\]
i.e.\ $a\leadsto a=1$. If $a\leq\bar b$ then because of $\bar b\vee a\leq\bar a\leadsto\bar a$ we have
\[
\bar a\odot\bar b=\bar a\odot(\bar b\vee a)\leq\bar a\wedge\bar a=\bar a
\]
whence
\[
a=\tilde{\bar a}\leq\widetilde{\bar a\odot\bar b}=a+b
\]
showing that $a+\widetilde{a+b}$ is defined. Now in case $a\leq\bar b$ the following are equivalent:
\begin{align*}
                    c & =\overline{a+\widetilde{a+b}}, \\
c+(a+\widetilde{a+b}) & =1, \\
(c+a)+\widetilde{a+b} & =1, \\
                  c+a & =a+b.
\end{align*}
Since
\[
(1\vee\bar a)\odot a=1\odot a=a\leq a\wedge a
\]
we have
\[
1=1\vee\bar a\leq a\rightarrow a,
\]
i.e.\ $a\rightarrow a=1$. If $b\leq\tilde a$ then because of $\tilde a\vee b\leq\tilde b\rightarrow\tilde b$ we have
\[
\tilde a\odot\tilde b=(\tilde a\vee b)\odot\tilde b\leq\tilde b\wedge\tilde b=\tilde b
\]
whence
\[
b=\bar{\tilde b}\leq\overline{\tilde a\odot\tilde b}=\widetilde{\bar a\odot\bar b}=a+b
\]
showing that $\overline{a+b}+b$ is defined. Now in case $b\leq\tilde a$, i.e.\ $a\leq\bar b$ the following are equivalent:
\begin{align*}
                   c & =\widetilde{\overline{a+b}+b}, \\
(\overline{a+b}+b)+c & =1, \\
\overline{a+b}+(b+c) & =1, \\
                 b+c & =a+b.
\end{align*}
This shows (P1). Now the following are equivalent:
\begin{align*}
& a\leq b\text{ holds in }\mathbb P(\mathbf Q), \\
& a+\tilde b\text{ is defined}, \\
& \bar a\odot b\text{ is defined}, \\
& a\leq b\text{ holds in }\mathbf Q.
\end{align*}
Since $(Q,\vee,\wedge)$ is a lattice and the partial order relations in $\mathbf Q$ and $\mathbb P(\mathbf Q)$ coincide, $\mathbb P(\mathbf Q)$ is a lattice pseudoeffect algebra.
\end{proof}

As in the case of effect algebras, also every good lattice pseudoeffect algebra can be reconstructed from its assigned quasiresiduated lattice.

\begin{theorem}
Let $\mathbf P$ be a good lattice pseudoeffect algebra. Then $\mathbb P(\mathbb Q(\mathbf P))=\mathbf P$.
\end{theorem}

\begin{proof}
Let
\begin{align*}
                      \mathbf P & =(P,+,\,\bar{}\,,\,\tilde{}\,,0,1)\text{ with lattice operations }\vee\text{ and }\wedge, \\
           \mathbb Q(\mathbf P) & =(P,\vee,\wedge,\odot,\rightarrow,\leadsto,0,1), \\
\mathbb P(\mathbb Q(\mathbf P)) & =(P,\oplus,{}^*,{}^+,0,1)
\end{align*}
and $a,b\in E$. Then
\begin{align*}
a^* & =a\rightarrow0=\bar a+(a\wedge0)=\bar a+0=\bar a, \\
a^+ & =a\leadsto0=(a\wedge0)+\tilde a=0+\tilde a=\tilde a.
\end{align*}
Moreover, the following are equivalent:
\begin{align*}
& a\oplus b\text{ is defined}, \\
& a\leq\bar b\text{ in }\mathbb Q(\mathbf P), \\
& a\leq\bar b\text{ in }\mathbf P
\end{align*}
and in this case
\[
a\oplus b=(a^*\odot b^*)^+=\widetilde{\bar a\odot\bar b}=\overline{\tilde a\odot\tilde b}=a+b.
\]
\end{proof}

Authors' addresses:

Ivan Chajda \\
Palack\'y University Olomouc \\
Faculty of Science \\
Department of Algebra and Geometry \\
17.\ listopadu 12 \\
771 46 Olomouc \\
Czech Republic \\
ivan.chajda@upol.cz

Helmut L\"anger \\
TU Wien \\
Faculty of Mathematics and Geoinformation \\
Institute of Discrete Mathematics and Geometry \\
Wiedner Hauptstra\ss e 8-10 \\
1040 Vienna \\
Austria, and \\
Palack\'y University Olomouc \\
Faculty of Science \\
Department of Algebra and Geometry \\
17.\ listopadu 12 \\
771 46 Olomouc \\
Czech Republic \\
helmut.laenger@tuwien.ac.at
\end{document}